\documentclass [twoside,reqno,12pt] {amsart}

\usepackage{amsfonts}
\usepackage{amssymb}
\usepackage{a4}

\newtheorem{thm}{Theorem}[section]
\newtheorem{cor}[thm]{Corollary}

\newtheorem{prop}[thm]{Proposition}

\theoremstyle{definition}

\theoremstyle{remark}
\newtheorem{rem}{Remark}[section]

\numberwithin{equation}{section}

\renewcommand{\Re}{\hbox{Re}\,}
\renewcommand{\Im}{\hbox{Im}\,}

\newcommand{\C}{\mathbb{C}}

\newcommand{\N}{\mathbb{N}}

\newcommand{\R}{\mathbb{R}}

\newcommand{\Z}{\mathbb{Z}}

\def\tilde{\widetilde}
\def \bfo {\begin {eqnarray*} }
\def \efo {\end {eqnarray*} }
\def \ba {\begin {eqnarray*} }
\def \ea {\end {eqnarray*} }
\def \beq {\begin {eqnarray}}
\def \eeq {\end {eqnarray}}

\def \p {\partial}

\def\tilde{\widetilde}
\def \bfo {\begin {eqnarray*} }
\def \efo {\end {eqnarray*} }
\def \ba {\begin {eqnarray*} }
\def \ea {\end {eqnarray*} }
\def \beq {\begin {eqnarray}}
\def \eeq {\end {eqnarray}}

\def \p {\partial}

\begin{document}

 \title[A Borg-Levinson theorem for  elliptic operators]{A Borg-Levinson theorem for higher order elliptic operators}

\author[Krupchyk]{Katsiaryna Krupchyk}

\address
        {K. Krupchyk, Department of Mathematics and Statistics \\
         University of Helsinki\\
         P.O. Box 68 \\
         FI-00014   Helsinki\\
         Finland}

\email{katya.krupchyk@helsinki.fi}

\author[P\"aiv\"arinta]{Lassi P\"aiv\"arinta}

\address
        {L. P\"aiv\"arinta, Department of Mathematics and Statistics \\
         University of Helsinki\\
         P.O. Box 68 \\
         FI-00014   Helsinki\\
         Finland}

\email{Lassi.Paivarinta@rni.helsinki.fi}

\maketitle

\begin{abstract} 

We establish the Borg-Levinson theorem for  elliptic operators of higher order with constant coefficients. The case of incomplete spectral data is also considered.

\end{abstract}

\section{Introduction}

The classical one-dimensional Borg-Levinson theorem concerns the unique determination of a potential in the Sturm-Liouville problem on a bounded interval from the spectral data.  Already in the original paper by Borg \cite{Borg_1946},  it was shown that a single spectrum in general does not suffice to determine the potential uniquely. The pioneering contributions \cite{Borg_1946, Levinson_1949} have given rise to a fundamental body of results on inverse problems of spectral and scattering theory in dimension one. The monographs \cite{L_book,  M_book, PT_book} give a detailed account of the theory. 

The multidimensional analog of the Borg-Levinson theorem for the Schr\"odinger operator $-\Delta +q$, on a bounded domain, states that the potential $q\in L^\infty$ is uniquely determined  by the Dirichlet eigenvalues and the boundary values of the normal derivatives of the eigenfunctions. This result was established in  \cite{Nachman_Sylveser_Uhlmann_88} and \cite{Nov_1988}.  Stability estimates for this problem have subsequently been obtained in 
\cite{AlesSyl_1990}. More recently, the case of singular potentials was treated in \cite{Pai_Ser_2002}, see also \cite{Chan_1990}, and \cite{Serov} for the case of magnetic potentials.

For one-dimensional Schr\"odinger operators on a bounded interval,  there is a one-to-one correspondence between the potential and the pair of all eigenvalues and normal derivatives of the eigenfunctions, see \cite{PT_book}.  However, this is no longer true for multidimensional Schr\"odinger operators, as it was shown in \cite{Isozaki91}. Specifically, it was proven in the latter paper that for a multidimensional Schr\"odinger operator on a bounded domain, 
the knowledge of all large Dirichlet eigenvalues and the boundary values of the normal derivatives of the corresponding eigenfunctions, still suffices to recover the potential uniquely. This result is usually referred to as the Borg-Levinson theorem with incomplete data.

In the first part of the paper, under some suitable assumptions, we obtain a natural extension of the multidimensional Borg-Levinson theorem due to \cite{Nachman_Sylveser_Uhlmann_88, Nov_1988}, to elliptic operators of higher order with constant coefficients.  We also show that the Laplacian, as 
well as the polyharmonic operators, all satisfy these assumptions.  

The second part of the paper is devoted to the study of the inverse spectral problem for higher order elliptic operators with incomplete data. 
Our approach here is different from the one, developed in \cite{Isozaki91}, and is based on the fact that the linear span of the  products of solutions to the 
Schr\"odinger equation, which satisfy a finite number of linear constraints on the boundary, is  dense in $L^1$ on the domain. We have learned of this idea from \cite{Ramm1993}, 
and  exploit it here in the context of elliptic operators of higher order. In particular, we obtain a natural analogue of the result of \cite{Isozaki91}, in the case of the polyharmonic operator.

Among the works in the field of inverse problems, devoted to the consideration of higher order elliptic operators, we would like to mention the papers
\cite{Ikeh_1991}, \cite{Isakov91} and \cite{Liu_Ch_1996}.  The motivation for such studies comes, in particular, from various problems in physics and geometry, such as quantum field theory \cite{Esp_Kam_1999}, theory of thin elastic plates in mechanics \cite{Villa_book}, as well as conformal  geometry, including the study of the Paneitz-Branson operator \cite{Bran_1985}. 

Finally, we mention that related inverse spectral problems for second order elliptic operators on compact Riemannian manifolds have also been extensively studied. We refer to the papers \cite{Bel_Kur_1992,  KatKur_1998, KurLasWed_2005}, as well as to the monograph \cite{KKL_book} for a detailed exposition of this theory.

The plan of the paper is as follows. Section 2 is devoted to the description of the problem and the statement of results. The extension of the multidimensional Borg-Levinson theorem to higher order elliptic operators is obtained in Section 3, with applications given in Section 4. Finally, the case of incomplete spectral 
data is treated in Section 5.

\section{Statement of results}

Let $P=P(D)$ be an elliptic partial differential operator on $\R^n$, $n\ge 2$,  of order $2m$, $m\ge 1$, with constant real coefficients,
\[
P(D)=\sum_{|\alpha|\le 2m} a_{\alpha}D^\alpha, \quad a_\alpha\in\R, \quad D_j=-i\partial_{x_j},\quad j=1,\dots,n.
\]   
Since the operator $P(D)$ is elliptic, without loss of generality, we may assume that its principal symbol satisfies
\[
\sum_{|\alpha|= 2 m} a_{\alpha}\xi^\alpha>0, \quad 0\ne\xi\in\R^n. 
\]
Therefore, through an application of the Fourier transform, we see that
\[
(P\varphi,\varphi)_{L^2(\R^n)}\ge C_1\|\varphi\|_{H^m(\R^n)}^2 -C_2\|\varphi\|^2_{L^2(\R^n)}, \quad \varphi\in C^\infty_0(\R^n), \quad C_1>0,\quad C_2\in\R.
\]
Here $H^m(\R^n)$ is the standard Sobolev space on $\R^n$.  

Let $\Omega\subset \R^n$ be a bounded domain with a $C^\infty$-boundary. Associated to $\Omega$, we have the Sobolev spaces $H^s(\Omega)$ and $H^s(\p \Omega)$, $s\in \R$.  Let $\gamma$ and $\tilde\gamma$ be the Dirichlet and Neumann trace operators, respectively,  given by
\begin{align*}
&\gamma:H^{2m}(\Omega)\to \mathcal{H}^{0,m-1}(\p \Omega):= \prod_{j=0}^{m-1}H^{2m-j-1/2}(\p \Omega),\\
&\gamma u=(u|_{\p\Omega},\p_{\nu}u|_{\p \Omega},\dots,\p_{\nu}^{m-1}u|_{\p \Omega}),\\
&\tilde \gamma: H^{2m}(\Omega)\to \mathcal{H}^{m,2m-1}(\p \Omega):= \prod_{j=m}^{2m-1}H^{2m-j-1/2}(\p \Omega),\\
&\tilde \gamma u=(\p_{\nu}^{m}u|_{\p \Omega},\dots,\p_{\nu}^{2m-1}u|_{\p \Omega}),
\end{align*}
which are bounded and surjective, see \cite{Grubbbook2009}.  Here $\nu$ is the unit outer normal to the boundary $\p \Omega$. 

The Friedrichs extension of $P$, defined on $C_0^\infty(\Omega)$, still denoted by $P$, is a self-adjoint operator semi-bounded from below, with the domain
\[
\mathcal{D}(P)=\{u\in H^{2m}(\Omega):\gamma u=0\},
\]
see \cite{Grubbbook2009}. 
Let $q\in L^\infty(\Omega)$ be real-valued.  Then by the Kato-Rellich theorem \cite{Kato_book},  the operator $P+q$ is self-adjoint on the domain $\mathcal{D}(P)$, and the spectrum of $P+q$ is discrete, accumulating at $+\infty$, consisting of eigenvalues of finite multiplicity, 
\[
-\infty <\lambda_{1,q}\le \lambda_{2,q}\le \dots\le \lambda_{k,q}\to+\infty,\quad k\to +\infty.
\] 
Associated to the eigenvalues $\lambda_{k,q}$, we have the eigenfunctions $\varphi_{k,q}\in \mathcal{D}(P)$, which form an orthonormal basis in $L^2(\Omega)$. The eigenvalues $\lambda_{k,q}$ and the eigenfunctions $\varphi_{k,q}$ will be referred to as the Dirichlet eigenvalues, respectively, the Dirichlet eigenfunctions, of $P+q$.

Let 
\[
P(\xi)=\sum_{|\alpha|\le 2m}a_{\alpha}\xi^\alpha, \quad \xi\in\R^n
\]
be the full symbol of the operator $P$. We let 
$
P(\zeta)=\sum_{|\alpha|\le 2m}a_{\alpha}\zeta^\alpha$,  $\zeta\in\C^n$, stand for the holomorphic continuation to $\C^n$.  
Following 
\cite{Hormander_book_II}, we set 
\[
\tilde P(\xi)=(\sum_{|\alpha|\ge 0}|P^{(\alpha)}(\xi)|^2)^{1/2}, \quad P^{(\alpha)}(\xi)=\p^{\alpha}_{\xi} P(\xi),\quad  \alpha\in \N^{n},
\]

In order to state our results, 
we shall introduce the following assumptions on $P(\xi)$:
\begin{itemize}
\item[\textbf{(A1)}] There exists a non-empty open subset
$U\subset \R^n$ and $\lambda_0>0$, such that for any $\xi\in U$ and any $\lambda\le -\lambda_0$, there are   $\zeta_1, \zeta_2\in \C^n$ such that 
\[
P(\zeta_j)=\lambda,\quad j=1,2,\quad \xi=\zeta_1-\overline{\zeta_2}.  
\]
\item[\textbf{(A2)}] Let $L_\zeta(\xi)=P(\xi+\zeta)-P(\zeta)$ and $P^{-1}(\lambda)=\{\zeta\in \C^n:P(\zeta)=\lambda\}$.  Assume that
\[
\sup_{\xi\in \R^n,\zeta\in P^{-1}(\lambda)}\frac{1}{\tilde L_\zeta(\xi)}\to 0,\quad \lambda\to -\infty. 
\]
\end{itemize}
Here we may notice that if $\lambda<0$ with $|\lambda|$ large enough, then $P^{-1}(\lambda)\cap\R^n=\emptyset$.

We have the following generalization of an $n$-dimensional Borg-Levinson theorem \cite{Nachman_Sylveser_Uhlmann_88}. 

\begin{thm}
\label{thm_main}

Assume that \emph{(A1)} and  \emph{(A2)} hold.  Let $q_1,q_2\in L^\infty(\Omega)$ be real-valued and $\varphi_{k,q_1}$ be an orthonormal basis in $L^2(\Omega)$ of the Dirichlet eigenfunctions of $P+q_1$. Furthermore, assume that  the Dirichlet eigenvalues $\lambda_{k,q_j}$ of $P+q_j$ satisfy
\begin{equation}
\label{eq_spec_d_1}
\lambda_{k,q_1}=\lambda_{k,q_2}, \quad k=1,2,\dots,
\end{equation}
and that there exists an orthonormal basis in $L^2(\Omega)$ of the Dirichlet eigenfunctions $\varphi_{k,q_2}$ of $P+q_2$ such that
\begin{equation}
\label{eq_spec_d_2}
\tilde \gamma\varphi_{k,q_1}=\tilde \gamma\varphi_{k,q_2},\quad k=1,2,\dots. 
\end{equation}
Then $q_1=q_2$. 
\end{thm}

It will be shown in Section \ref{sec_applications} that  
the conditions (A1) and (A2) are satisfied for the Laplace operator $P=-\Delta$ and the polyharmonic operator $P=(-\Delta)^m$, $m\ge 2$. Let us also notice that the assumptions (A1) and (A2) are similar to those, which occur in \cite{Isakov91}. In the latter work, such  assumptions  are introduced in order to guarantee the completeness of the products of solutions of the equation $(P+q)u=0$ for general constant coefficient partial differential operators $P$.

In our approach to Theorem \ref{thm_main} we have been inspired  by the exposition of the classical multidimensional Borg-Levinson theorem, given in 
\cite{Choulli_book}. Rather than using scattering solutions of $(P+q-\lambda)u=0$, constructed in all of $\R^n$, for $\lambda>0$ large enough as in \cite{Nachman_Sylveser_Uhlmann_88},   following \cite{Choulli_book}, we shall make use of complex geometric optics solutions, constructed in $\Omega$ for $\lambda<0$, $|\lambda|$ sufficiently large. 

The second main result of this work is concerned with the Borg-Levinson problem with incomplete spectral data in the case when  
$P=(-\Delta)^m$ is the polyharmonic operator. It can be viewed as an analog of the result of \cite{Isozaki91} valid in the case of the Laplacian.

\begin{thm}
\label{thm_isozaki_bi}
Let $q_1,q_2\in L^\infty(\Omega)$ be real-valued and $\varphi_{k,q_1}$ be an orthonormal basis in $L^2(\Omega)$
of the Dirichlet eigenfunctions of the operator $(-\Delta)^m+q_1$, $m\ge 2$.  Assume that there exists  an integer  $N>0$ such that the Dirichlet eigenvalues $\lambda_{k,q_j}$ of $(-\Delta)^m+q_j$ satisfy
\[
\lambda_{k,q_1}=\lambda_{k,q_2},\quad \forall k>N,
\]
and there exists an orthonormal basis in $L^2(\Omega)$ of the Dirichlet  eigenfunctions $\varphi_{k,q_2}$ of $(-\Delta)^m+q_2$ such that
\[
\tilde \gamma\varphi_{k,q_1}=\tilde \gamma\varphi_{k,q_2},\quad \forall k>N. 
\]
Then $q_1=q_2$.
\end{thm}

We shall finish this section by introducing, for future reference, the Dirichlet--to--Neumann map, associated to the operator $P+q_j-\lambda$.   When doing so, suppose that $\lambda_0>0$ is large enough so that for any $\lambda\le -\lambda_0$,  zero is not in the spectrum of the operator $P+q_j-\lambda$, $j=1,2$, equipped with the domain $\mathcal{D}(P)$. Here and in what follows the domain $\mathcal{D}(P)$ will be provided with the graph norm, which is easily seen to be equivalent to $\|\cdot\|_{H^{2m}(\Omega)}$. 
For any $\lambda<-\lambda_0$ and any $f\in\mathcal{H}^{0,m-1}(\p\Omega)$, the Dirichlet boundary problem
\begin{equation}
\label{eq_bvp}
\begin{aligned}
(P+q_j-\lambda)u &=0,\quad \textrm{in}\quad \Omega,\\
\gamma u &=f,\quad \textrm{on}\quad \p \Omega,
\end{aligned} 
\end{equation}
has a  unique solution $u_{q_j,f}(\lambda)\in H^{2m}(\Omega)$ and
\[
\|u_{q_j,f}(\lambda)\|_{H^{2m}(\Omega)}\le C\|f\|_{\mathcal{H}^{0,m-1}(\p \Omega)}, 
\]
see \cite{Grubbbook2009}. 
Here $C>0$ may depend on $\lambda$. 
Thus, for any $\lambda\le -\lambda_0$, we define the Dirichlet--to--Neumann map by
\[
\Lambda_{q_j}(\lambda)(f)=\tilde \gamma u_{q_j,f}(\lambda),
\]
which is a bounded map
\[
\Lambda_{q_j}(\lambda):
\mathcal{H}^{0,m-1}(\p \Omega)\to \mathcal{H}^{m,2m-1}(\p \Omega).
\]

\section{Proof of Theorem \ref{thm_main}}

Let us start this section by providing  a general outline of the proof of Theorem \ref{thm_main}.  
The first step is to construct complex geometric optics solutions of the equations 
\begin{equation}
\label{eq_201}
(P+q_j-\lambda)u_j(\lambda)=0 \quad  \textrm{in}\quad \Omega, \quad j=1,2.
\end{equation}
This is possible to achieve thanks to the assumption (A2), combined with the general estimate for a right inverse in $L^2(\Omega)$ of the operator $P$, stated in Theorem \ref{thm_isakov} below. One subsequently uses the assumption (A1) to establish the density of the linear span of products of solutions $u_1(\lambda) \overline{u_2(\lambda)}$ of the equations \eqref{eq_201}, for $\lambda<0$ with $|\lambda|$ sufficiently large. The density result implies that in order to establish that $q_1=q_2$, it suffices to show  the equality of the corresponding Dirichlet--to--Neumann maps $\Lambda_{q_1}(\lambda)=\Lambda_{q_2}(\lambda)$, for all $\lambda<0$ with $|\lambda|$ large enough. 
The latter is done using general arguments, involving only the assumptions  \eqref{eq_spec_d_1} and \eqref{eq_spec_d_2} in Theorem \ref{thm_main}.

We shall now proceed with the detailed proof of Theorem \ref{thm_main}. Let us first  recall the following result due to \cite{Isakov91},  where it is 
obtained as a consequence of the general theory of \cite{Hormander_book_II}.

\begin{thm} 
\label{thm_isakov}

Let $P$ be a partial differential operator on $\R^n$  with constant coefficients and let $\Omega$ be a bounded domain in $\R^n$. Then there exists a bounded linear operator $E\in \mathcal{L}(L^2(\Omega))$ such that
\[
PEf=f, \quad \textrm{for all}\ f\in L^2(\Omega),
\]
and for any partial differential operator $Q$ with constant coefficients, we have
\[
\|Q(D)E\|_{\mathcal{L}(L^2(\Omega))}\le C\sup_{\xi\in \R^n}\frac{\tilde Q(\xi)}{\tilde P(\xi)},
\]
where $C>0$ depends only on $n$, $\Omega$, and the order of $P$. 
\end{thm}

Let $\tilde \Omega\supset\supset \Omega$ be an open bounded subset of $\R^n$ with $C^\infty$-smooth boundary. 
Let $q\in L^\infty(\Omega)$ and let us extend $q$ to $\tilde \Omega$ by setting $q=0$ in $\tilde \Omega\setminus\Omega$.  
The following result is a generalization of \cite[Proposition 2.10]{Choulli_book} to higher order elliptic operators, see also \cite{Isakov91}.

\begin{prop}
\label{prop_geometric_optics}
Assume that \emph{(A2)} holds. Then there exists $\lambda_0>0$ such that for any $\lambda<-\lambda_0$ and any $\zeta\in P^{-1}(\lambda)$, there are solutions
\[
u_{\lambda,\zeta}=e^{i\zeta\cdot x}(1+w_{\lambda,\zeta})\in L^2(\tilde \Omega)
\]
to the equation $(P+q-\lambda)u=0$ in $\tilde \Omega$
with $\|w_{\lambda,\zeta}\|_{L^2(\tilde \Omega)}\to 0$ as $\lambda\to-\infty$. 

\end{prop}

\begin{proof} 
Set $u_{\lambda,\zeta}=e^{i\zeta\cdot x}(1+w_{\lambda,\zeta})$. Then using Leibniz' formula, 
\[
P(D)(uv)=\sum_{|\alpha|\ge 0} \frac{1}{\alpha!}P^{(\alpha)}(D)u D^{\alpha}v,
\] 
and the fact that $P(\zeta)=\lambda$, we get
\begin{align*}
P(D)u_{\lambda,\zeta}=\lambda u_{\lambda,\zeta}+\sum_{|\alpha|\ge 1} \frac{1}{\alpha!}P^{(\alpha)}(\zeta)e^{i\zeta\cdot x} D^{\alpha}w_{\lambda,\zeta}.
\end{align*}
In order that  $(P+q-\lambda)u_{\lambda,\zeta}=0$, the correction $w_{\lambda,\zeta}$ should satisfy
\begin{equation}
\label{eq_w}
L_{\zeta}(D)w_{\lambda,\zeta}=-q(1+w_{\lambda,\zeta}),
\end{equation}
where
\[
L_{\zeta}(D)=\sum_{|\alpha|\ge 1} \frac{1}{\alpha!}P^{(\alpha)}(\zeta)D^{\alpha}=P(\zeta+D)-P(\zeta).
\]
By Theorem \ref{thm_isakov}, there is an operator $E_\zeta\in \mathcal{L}(L^2(\tilde \Omega))$ such that
\[
L_{\zeta}(D)E_\zeta f=f,
\]
for any $f\in L^2(\tilde \Omega)$, and 
\[
\|E_\zeta \|_{\mathcal{L}(L^2(\tilde \Omega))}\le C\sup_{\xi\in\R^n}\frac{1}{\tilde L_\zeta(\xi)}\le 
C\sup_{\xi\in\R^n,\zeta\in P^{-1}(\lambda)}\frac{1}{\tilde L_\zeta(\xi)},
\]
where a constant $C>0$ depends only  on $m$, $n$, and $\tilde \Omega$.
The assumption (A2) implies that $\|E_\zeta \|_{\mathcal{L}(L^2(\tilde \Omega))}\to 0$, as $\lambda\to-\infty$, uniformly in $\zeta\in P^{-1}(\lambda)$. Hence, there exists $\lambda_0>0$ large enough such that  the map
\[
F_\zeta: L^2(\tilde \Omega)\to L^2(\tilde \Omega),\quad f\mapsto E_\zeta(-q(1+f))
\]
is a contraction for any $\zeta\in P^{-1}(\lambda)$ and any $\lambda\le -\lambda_0$. Thus, $F_\zeta$ has a unique fixed point $w_{\lambda,\zeta}\in L^2(\tilde \Omega)$, and therefore, \eqref{eq_w} holds. Furthermore, for $\lambda<0$, $|\lambda|$ large, we have
\begin{align*}
\|w_{\lambda,\zeta}\|_{L^2(\tilde\Omega)}&\le \|E_\zeta q\|_{L^2(\tilde \Omega)}+\|E_\zeta qw_{\lambda,\zeta}\|_{L^2(\tilde \Omega)}\\
&\le \|E_\zeta \|_{\mathcal{L}(L^2(\tilde \Omega))}\|q\|_{L^2(\tilde \Omega)}+\|E_\zeta \|_{\mathcal{L}(L^2(\tilde \Omega))}\|q\|_{L^\infty(\tilde \Omega)}\|w_{\lambda,\zeta}\|_{L^2(\tilde \Omega)}\\
&\le \|E_\zeta \|_{\mathcal{L}(L^2(\tilde \Omega))}\|q\|_{L^2(\tilde \Omega)}+\frac{1}{2}\|w_{\lambda,\zeta}\|_{L^2(\tilde\Omega)}.
\end{align*}
The claim follows. 

\end{proof}

The solutions $u_{\lambda,\zeta}$ constructed in Proposition \ref{prop_geometric_optics} will be referred to as the complex geometric optics solutions. See \cite{SylUhl_87} for the original construction of such solutions in the case of the Laplacian.

We can now obtain the following density result. 

\begin{prop}
\label{prop_completeness}
Suppose that the assumptions \emph{(A1)} and \emph{(A2)} hold. Then there exists $\lambda_0>0$ such that 
the set
\[
S(q_1,q_2,\lambda_0)
=\emph{span} \cup_{\lambda\le -\lambda_0}F(q_1,q_2,\lambda),
\]
where 
\begin{align*}
F(q_1,&q_2,\lambda)\\
&=\{u_{q_1}(\lambda) \overline{u_{q_2}(\lambda)}: u_{q_j}(\lambda)\in H^{2m}(\Omega), (P+q_j-\lambda)u_{q_j}(\lambda)=0 \ \emph{in} \  \Omega,j=1,2\}, 
\end{align*}
 is dense in $L^1(\Omega)$.
\end{prop}

\begin{proof}
Assume that $f\in L^\infty(\Omega)$ is such  that
\[
\int_{\Omega}fgdx=0,
\]
for any $g\in S(q_1,q_2,\lambda_0)$. Let $\lambda_0$ be the largest of the values $\lambda_0$, occurring  in the assumption (A1) and 
Proposition \ref{prop_geometric_optics}. Then for any $\lambda\le -\lambda_0$ and any $\xi\in U\subset\R^n$, there are $\zeta_j\in P^{-1}(\lambda)$
such that $\xi=\zeta_1-\overline{\zeta_2}$. Let $u_{\lambda,\zeta_j}\in L^2(\tilde \Omega)$ be the  complex geometric optics solutions, constructed in Proposition \ref{prop_geometric_optics}. Then $Pu_{\lambda,\zeta_j}\in L^2(\tilde \Omega)$, and 
 by elliptic regularity,  $u_{\lambda,\zeta_j}\in H^{2m}(\Omega)$. As $\overline{e^{i\zeta_2\cdot x}}=e^{-i\overline{\zeta_2}\cdot x}$, we have
\[
\int_\Omega f u_{\lambda,\zeta_1}\overline{u_{\lambda,\zeta_2}}dx=\int_\Omega e^{i\xi\cdot x}f dx+\int_\Omega gdx=0,
\]
with $g=e^{i\xi\cdot x}(w_{\lambda,\zeta_1}+\overline{w_{\lambda,\zeta_2}}+w_{\lambda,\zeta_1}\overline{w_{\lambda,\zeta_2})}f$. 
As
\[
\bigg|\int_\Omega gdx\bigg|\le C(\|w_{\lambda,\zeta_1}\|_{L^2(\Omega)}+\|w_{\lambda,\zeta_2}\|_{L^2(\Omega)}+\|w_{\lambda,\zeta_1}\|_{L^2(\Omega)}\|w_{\lambda,\zeta_2}\|_{L^2(\Omega)})\to 0, 
\]
as $\lambda\to-\infty$, we have
\[
\int_\Omega e^{i\xi\cdot x}f dx=0,\quad \forall\xi\in U.
\]
Since $\Omega$ is bounded, the left hand side is a real analytic function on $\R^n$, which has been shown 
to vanish on the open set $U$. Thus, $f=0$. This proves that $S(q_1,q_2,\lambda_0)$ is dense in $L^1(\Omega)$.

\end{proof}

\begin{rem}

Notice that when $q_1=q_2=0$, the conclusion of Proposition \ref{prop_completeness} remains valid, assuming that only assumption (A1) holds. 

\end{rem}

In what follows, we shall need the generalization of  Green's formula, given in \cite{Agmon_book, Folland_book},

\begin{equation}
\label{eq_green}
(Pu,v)_{L^2(\Omega)}-(u,Pv)_{L^2(\Omega)}=\sum_{j=0}^{2m-1}\int_{\p \Omega} N_{2m-1-j}(u)\overline{\p_{\nu}^j v}dS.
\end{equation}
Here $N_k $ is a linear differential operator of order $k$, defined in a neighborhood of $\p \Omega$, which contains the term $\p_\nu^k$ with a non-vanishing coefficient and $dS$ is the surface measure on $\p \Omega$.

The following result is a generalization of \cite[Theorem 1.5]{Nachman_Sylveser_Uhlmann_88} to higher order elliptic operators.

\begin{prop}
\label{prop_q_1_q2}
Assume that \emph{(A1)} and \emph{(A2)} hold. 
Let $q_1,q_2\in L^\infty(\Omega)$ and assume that there exists $\lambda_0>0$  such that the corresponding Dirichlet--to--Neumann maps satisfy
$\Lambda_{q_1}(\lambda)=\Lambda_{q_2}(\lambda)$ for any $\lambda\le -\lambda_0$. Then $q_1=q_2$. 
\end{prop}

\begin{proof}
Let $f\in \mathcal{H}^{0,m-1}(\p \Omega)$ and $u_{q_j,f}(\lambda)\in H^{2m}(\Omega)$ be solutions to 
\eqref{eq_bvp}.  Set
\[
u=u_{q_2,f}(\lambda)-u_{q_1,f}(\lambda). 
\]
Then, for $\lambda\le -\lambda_0$,
\begin{align*}
(P+q_2-\lambda)u&=(q_1-q_2)u_{q_1,f}(\lambda),\quad \textrm{in}\quad \Omega\\
\gamma u&=0, \quad \textrm{on}\quad \p\Omega.
\end{align*}
Since $\Lambda_{q_1}(\lambda)=\Lambda_{q_2}(\lambda)$ , we have
\[
\tilde \gamma u=0, \quad \textrm{on}\quad \p\Omega
\]
For any $v\in H^{2m}( \Omega)$ solving  $(P+q_2-\lambda)v=0$ in $\Omega$, by an application of the Green's formula
\eqref{eq_green}, we get
\[
\int_{\Omega}(q_1-q_2)u_{q_1,f}(\lambda)\overline{v}dx=0. 
\]
Proposition \ref{prop_completeness} implies that $q_1=q_2$. This completes the proof.

\end{proof}

Notice that 
\[
\mathcal{H}^{m,2m-1}(\p \Omega)\subset \mathcal{H}^{m,2m-1,\varepsilon}(\p \Omega):=\prod_{j=m}^{2m-1}H^{2m-\varepsilon-j-1/2}(\p \Omega),\quad \forall\varepsilon>0.
\]

\begin{prop}
\label{prop_norm_dirichlet-to-neumann}
For any small $\varepsilon>0$,
\[
\|\Lambda_{q_1}(\lambda)-\Lambda_{q_2}(\lambda)\|_{\mathcal{L}(\mathcal{H}^{0,m-1}(\p \Omega),\mathcal{H}^{m,2m-1,\varepsilon}(\p \Omega))}\to 0,\quad \lambda\to-\infty. 
\]
\end{prop}

\begin{proof} Let $u_{q_j,f}(\lambda)\in H^{2m}(\Omega)$ be a solution to \eqref{eq_bvp}, $j=1,2$. Then $u=u_{q_1,f}(\lambda)-u_{q_2,f}(\lambda)$ solves the problem
\begin{equation}
\label{eq_lim_6}
\begin{aligned}
(P+q_1-\lambda)u&=(q_2-q_1)u_{q_2,f}(\lambda),\quad \textrm{in}\quad \Omega\\
\gamma u&=0, \quad \textrm{on}\quad \p\Omega.
\end{aligned}
\end{equation}
Thus, $u\in \mathcal{D}(P)$, and therefore, 
\[
(Pu,u)_{L^2(\Omega)}\ge -C\|u\|_{L^2(\Omega)}^2,\quad C\ge 0.
\]
Assume that $\lambda<0$ and $|\lambda|$ is so large that
\[
((P+q_j-\lambda)u,u)_{L^2(\Omega)}\ge \frac{|\lambda|}{2}\|u\|_{L^2(\Omega)}^2,\quad j=1,2. 
\]
Then
\begin{equation}
\label{eq_lim_1}
\begin{aligned}
\frac{|\lambda|}{2}\|u\|_{L^2(\Omega)}^2&\le ((q_2-q_1)u_{q_2,f}(\lambda),u)_{L^2(\Omega)}\\
&\le \|q_2-q_1\|_{L^\infty(\Omega)}\|u_{q_2,f}(\lambda)\|_{L^2(\Omega)}\|u\|_{L^2(\Omega)}.
\end{aligned}
\end{equation}
This implies that
\begin{equation}
\label{eq_lim_3}
\|u\|_{L^2(\Omega)}\le \frac{C}{|\lambda|}\|u_{q_2,f}(\lambda)\|_{L^2(\Omega)},
\end{equation}
where the constant $C>0$ is independent of $\lambda$.

Let $\tilde \lambda>0$ be large enough but fixed so that zero is not in the spectrum of the operator $P+\tilde\lambda$, equipped with the domain $\mathcal{D}(P)$. Then $u_{q_2,f}(\lambda)=v_0+v_1$, where $v_0$ is a solution to the problem
\begin{equation}
\label{eq_lim_2}
\begin{aligned}
(P+\tilde \lambda)v_0&=0,\quad \textrm{in}\quad \Omega,\\
\gamma v_0&=f, \quad \textrm{on}\quad \p\Omega,
\end{aligned}
\end{equation}
and $v_1$ is a solution to the problem
\begin{align*}
(P+q_2-\lambda)v_1&=(\tilde \lambda +\lambda-q_2)v_0,\quad \textrm{in}\quad \Omega\\
\gamma v_1&=0, \quad \textrm{on}\quad \p\Omega.
\end{align*}
Thus, $v_1\in \mathcal{D}(P)$, and, similarly to \eqref{eq_lim_1},  we get
\[
\frac{|\lambda|}{2}\|v_1\|_{L^2(\Omega)}\le\|\tilde \lambda+\lambda-q_2\|_{L^\infty}\|v_0\|_{L^2(\Omega)}. 
\]
This yields that
\begin{equation}
\label{eq_lim_4}
\|v_1\|_{L^2(\Omega)}\le C \|v_0\|_{L^2(\Omega)},
\end{equation}
with $C>0$ and $C$ does not depend on $\lambda$.  It follows from \eqref{eq_lim_2} that
\begin{equation}
\label{eq_lim_5}
\|v_0\|_{L^2(\Omega)}\le \|v_0\|_{H^{2m}(\Omega)} \le C\|f\|_{\mathcal{H}^{0,m-1}(\p \Omega)},
\end{equation}
where $C>0$ is independent of $\lambda$.  Using \eqref{eq_lim_3}, \eqref{eq_lim_4} and \eqref{eq_lim_5}, we get
\begin{equation}
\label{eq_lim_7}
\|u_{q_2,f}(\lambda)\|_{L^2(\Omega)}\le C\|f\|_{\mathcal{H}^{0,m-1}(\p \Omega)}
\end{equation}
and
\begin{equation}
\label{eq_lim_8}
\|u\|_{L^2(\Omega)}\le\frac{C}{|\lambda|}\|f\|_{\mathcal{H}^{0,m-1}(\p \Omega)},
\end{equation}
where $C>0$ is independent of $\lambda$.

 Let us now proceed to derive an estimate for the norm of $u$ in $H^{2m}(\Omega)$. 
Let $\tilde \lambda>0$ be fixed and large enough so that zero is not in  the spectrum of the operator $P+q_1+\tilde\lambda$, equipped with the domain $\mathcal{D}(P)$. Then \eqref{eq_lim_6} implies that
\begin{align*}
(P+q_1+\tilde \lambda)u&=(q_2-q_1)u_{q_2,f}(\lambda)+(\tilde \lambda+\lambda)u,\quad \textrm{in}\quad \Omega,\\
\gamma u&=0, \quad \textrm{on}\quad \p\Omega,
\end{align*}
and 
\[
\|u\|_{H^{2m}(\Omega)}\le C(|\tilde\lambda|+|\lambda|)\|u\|_{L^2(\Omega)}+\|q_2-q_1\|_{L^\infty(\Omega)}\|u_{q_2,f}(\lambda)\|_{L^2(\Omega)}.
\]
It follows from \eqref{eq_lim_7} and \eqref{eq_lim_8} that
\begin{equation}
\label{eq_lim_9}
\|u\|_{H^{2m}(\Omega)}\le C\|f\|_{\mathcal{H}^{0,m-1}(\p \Omega)},
\end{equation}
where $C>0$ is independent of $\lambda$.

By an interpolation property of the Sobolev norms, see \cite{Grubbbook2009}, we get
\[
\|u\|_{H^s(\Omega)}\le C\|u\|_{L^2(\Omega)}^{1-\frac{s}{2m}}\|u\|_{H^{2m}(\Omega)}^{\frac{s}{2m}}, \quad 0\le s\le 2m.
\]
It follows from \eqref{eq_lim_8} and \eqref{eq_lim_9} that
\[
\|u\|_{H^s(\Omega)}\le \frac{C}{|\lambda|^{1-\frac{s}{2m}}}\|f\|_{\mathcal{H}^{0,m-1}(\p \Omega)}, \quad 0\le s\le 2m,
\]
where $C>0$ is independent of $\lambda$. Thus, 
\[
\|\p_\nu^j u|_{\p \Omega}\|_{H^{2m-\varepsilon-j-1/2}(\p \Omega)}\le C\|u\|_{H^{2m-\varepsilon}(\Omega)}\le  \frac{C}{|\lambda|^{\frac{\varepsilon}{2m}}}\|f\|_{\mathcal{H}^{0,m-1}(\p \Omega)},
\]
where $m\le j\le 2m-1$ and $\varepsilon >0$ small.  Therefore,
\[
\|\Lambda_{q_1}(\lambda)f-\Lambda_{q_2}(\lambda)(f)\|_{\mathcal{H}^{m,2m-1,\varepsilon}(\p\Omega)}\le  \frac{C}{|\lambda|^{\frac{\varepsilon}{2m}}}\|f\|_{\mathcal{H}^{0,m-1}(\p \Omega)}.
\]
The claim follows.

\end{proof}

In the proof of the following proposition, we shall need some basic facts concerning the resolvent of the self-adjoint operator $P+q_j$. 
Let $\rho(P+q_j)\subset \C$ be the resolvent set of $P+q_j$.  Notice that if $\lambda_0>0$ is large enough, then 
\begin{equation}
\label{eq_rho_1}
\rho(P+q_j)\supset\{\lambda\in \R:\lambda\le -\lambda_0\}.
\end{equation} 
The resolvent
\[
R_{q_j}(\lambda):=(P+q_j-\lambda)^{-1}: \rho(P+q_j)\to \mathcal{L}(L^2(\Omega),\mathcal{D}(P))
\]
is holomorphic. Furthermore, for any $h\in L^2(\Omega)$ and any $\lambda\in \rho(P+q_j)$, we have
\[
R_{q_j}(\lambda)h=\sum_{k\ge 1}\frac{1}{\lambda_{k,q_j}-\lambda}(h,\varphi_{k,q_j})_{L^2(\Omega)}\varphi_{k,q_j},
\]
where the series converges in $L^2(\Omega)$.  

\begin{prop}
\label{prop_deriv_DN_map} Assume that the hypotheses of Theorem \emph{\ref{thm_main}} hold. Then for each $f\in \mathcal{H}^{0,m-1}(\p \Omega)$, the function $\lambda\mapsto \Lambda_{q_j}(\lambda)f$, $j=1,2$, is holomorphic in the region $\Re \lambda<0$, $|\Re \lambda|$ large enough, with values in $\mathcal{H}^{m,2m-1}(\p \Omega)$. 
Moreover, 
for all $l\in\N$ satisfying $(l-1)m>n$ and all $\lambda<0$, $|\lambda|$ large enough, we have
\[
\frac{d^l}{d\lambda^l}(\Lambda_{q_1}(\lambda)f-\Lambda_{q_2}(\lambda)f)=0, \quad \forall f\in \mathcal{H}^{0,m-1}(\p \Omega).
\]
\end{prop}

\begin{proof}

Assume, as we may, that $\lambda_0$ in the hypothesis (A1) is such that the inclusion \eqref{eq_rho_1} holds. Then for $f\in \mathcal{H}^{0,m-1}(\p \Omega)$ and $\lambda<-\lambda_0$, consider a solution $u_{q_j,f}(\lambda)\in H^{2m}(\Omega)$ to the problem \eqref{eq_bvp}.

Let  $\tilde \lambda>0$ be large enough but fixed such that zero is not in the spectrum $P+\tilde \lambda$, equipped with the domain $\mathcal{D}(P)$.   We have
\begin{equation}
\label{eq_der_1}
u_{q_j,f}(\lambda)=F-R_{q_j}(\lambda)(q_j-\lambda-\tilde \lambda)F,
\end{equation}
where $F$ is a solution to 
\begin{align*}
(P+\tilde \lambda)F&=0,\quad \textrm{in}\quad \Omega,\\
\gamma F&=f, \quad \textrm{on}\quad \p\Omega.
\end{align*}
It follows from  \eqref{eq_der_1} that  $u_{q_j,f}(\lambda)$ is a holomorphic function of $\lambda\in\rho(P+q_j)$ with values in $H^{2m}(\Omega)$. 
Differentiating the problem \eqref{eq_bvp} with respect to $\lambda$, $l$ times, we get
\begin{align*}
(P+q_{j}- \lambda)\frac{d^l}{d\lambda^l}u_{q_j,f}(\lambda)&=l\frac{d^{l-1}}{d\lambda^{l-1}}u_{q_j,f}(\lambda),\quad \textrm{in}\quad \Omega,\\
\gamma (\frac{d^l}{d\lambda^l}u_{q_j,f}(\lambda))&=0, \quad \textrm{on}\quad \p\Omega.
\end{align*}
Hence, using \eqref{eq_der_1}, we have
\begin{align*}
\frac{d^l}{d\lambda^l}u_{q_j,f}(\lambda)&=l R_{q_j}(\lambda)\frac{d^{l-1}}{d\lambda^{l-1}}u_{q_j,f}(\lambda)=\dots=l! R_{q_j}(\lambda)^l u_{q_j,f}(\lambda)\\
&=
l! R_{q_j}(\lambda)^lF-l!R_{q_j}(\lambda)^{l+1}(q_j-\lambda-\tilde \lambda)F.
\end{align*}
Here
\[
R_{q_j}(\lambda)^{l}F=
\sum_{k\ge 1}\frac{1}{(\lambda_{k,q_j}-\lambda)^{l}}(F,\varphi_{k,q_j})_{L^2(\Omega)}\varphi_{k,q_j}
\]
and
\[
R_{q_j}(\lambda)^{l+1}(q_j-\lambda-\tilde \lambda)F=
\sum_{k\ge 1}\frac{1}{(\lambda_{k,q_j}-\lambda)^{l+1}}(F,(q_j-\lambda-\tilde \lambda)\varphi_{k,q_j})_{L^2(\Omega)}\varphi_{k,q_j}.
\]
Thus,
\[
\frac{d^l}{d\lambda^l}u_{q_j,f}(\lambda)=-l!\sum_{k\ge 1}\frac{1}{(\lambda_{k,q_j}-\lambda)^{l+1}}(F,(q_j-\lambda_{k,q_j}-\tilde \lambda)\varphi_{k,q_j})_{L^2(\Omega)}\varphi_{k,q_j}
\]
with the convergence in $L^2(\Omega)$. 
Now using the generalization of Green's formula \eqref{eq_green}, we obtain 
\begin{align*}
(F,(q_j-\lambda_{k,q_j}-\tilde \lambda)\varphi_{k,q_j})_{L^2(\Omega)}&=-(F,(P+\tilde \lambda)\varphi_{k,q_j})_{L^2(\Omega)}
\\
&=-\sum_{i=0}^{2m-1}\int_{\p\Omega} \overline{N_{2m-1-i}(\varphi_{k,q_j})}\p_\nu^i FdS,
\end{align*}
where $N_{2m-1-i}$ is a linear differential operator of order $2m-1-i$ in a neighborhood of $\p \Omega$, which contains the term $\p_\nu^{2m-1-i}$ with a non-vanishing coefficient. Since $\p_\nu^i \varphi_{k,q_j}|_{\p \Omega}=0$, $i=0,1,\dots m-1$, we have $N_{2m-1-i}(\varphi_{k,q_j})|_{\p \Omega}=0$ for $i\ge m$. Thus, 
\[
(F,(q_j-\lambda_{k,q_j}-\tilde \lambda)\varphi_{k,q_j})_{L^2(\Omega)}=-\sum_{i=0}^{m-1}\int_{\p\Omega} \overline{N_{2m-1-i}(\varphi_{k,q_j})}f_i dS,
\]
where $f=(f_0,\dots, f_{m-1})=\gamma F$. Hence,
\begin{equation}
\label{eq_lambda_der}
\frac{d^l}{d\lambda^l}u_{q_j,f}(\lambda)=l!\sum_{k\ge 1}\frac{1}{(\lambda_{k,q_j}-\lambda)^{l+1}}
\bigg(\sum_{i=0}^{m-1}\int_{\p\Omega} \overline{N_{2m-1-i}(\varphi_{k,q_j})}f_i dS\bigg)
\varphi_{k,q_j}.
\end{equation}
The series in the right hand side of \eqref{eq_lambda_der} converges in $L^2(\Omega)$ for any $\lambda\le -\lambda_0$. Let us show that it  also converges in $H^{2m}(\Omega)$ for $l$ large enough. Indeed, by the standard consequence of the Weyl law, the eigenvalues of the elliptic operator operator $P+q_j$ of order $m$ have the following asymptotics,
\[
\lambda_{k,q_j}\sim k^{m/n}, \quad k\to +\infty,
\]
in the sense that there exist constant $C_1,C_2>0$ such that $C_1 k^{m/n}\le \lambda_{k,q_j}\le C_2 k^{m/n}$, for all $k=1,2,\dots$,
see \cite{Safarov_Vassiliev_book_1997}. Hence, for large $k$, using that $\|\varphi_{k,q_j}
\|_{H^{2m}(\Omega)}\le C\lambda_k$,  we get
\begin{align*}
\|\frac{1}{(\lambda_{k,q_j}-\lambda)^{l+1}}(F,(q_j-\lambda_{k,q_j}-\tilde \lambda)\varphi_{k,q_j})_{L^2(\Omega)}\varphi_{k,q_j}
\|_{H^{2m}(\Omega)}\\
\le C k^{-\frac{(l+1)m}{n}}\|F\|_{L^2(\Omega)}\|q_j-\lambda_{k,q_j}-\tilde \lambda\|_{L^\infty(\Omega)}\|\varphi_{k,q_j}
\|_{H^{2m}(\Omega)} \le C k^{\frac{(1-l)m}{n}}.
\end{align*}

Therefore, the series in the right hand side of \eqref{eq_lambda_der} converges in $H^{2m}(\Omega)$ for $l$ satisfying $(l-1)m>n$.  

Furthermore, for all $l\in\N$ satisfying $(l-1)m>n$ and $r=m,\dots,2m-1$, we have
\[
\p_\nu^r\frac{d^l}{d\lambda^l}u_{q_j,f}(\lambda)=l!\sum_{k\ge 1}\frac{1}{(\lambda_{k,q_j}-\lambda)^{l+1}}
\bigg(\sum_{i=0}^{m-1}\int_{\p\Omega} \overline{N_{2m-1-i}(\varphi_{k,q_j})}f_i dS\bigg)
\p_\nu^r\varphi_{k,q_j}.
\]
From this and from the assumption  $\tilde \gamma \varphi_{k,q_1}=\tilde \gamma\varphi_{k,q_2}$, $k=1,2,\dots$, it  then follows
 for all $l\in\N$ satisfying $(l-1)m>n$ that 
\[
\frac{d^l}{d\lambda^l}(\Lambda_{q_1}(\lambda)f-\Lambda_{q_2}(\lambda)f)=0, \quad \forall f\in \mathcal{H}^{0,m-1}(\p \Omega).
\]
This completes the proof. 
\end{proof}

It is now easy to obtain the statement of Theorem \ref{thm_main}.
It follows from Proposition \ref{prop_deriv_DN_map} that $\Lambda_{q_1}(\lambda)-\Lambda_{q_2}(\lambda)$ is a polynomial in $\lambda$. Combining this with Proposition \ref{prop_norm_dirichlet-to-neumann}, we conclude that $\Lambda_{q_1}(\lambda)=\Lambda_{q_2}(\lambda)$ for all $\lambda<0$ with $|\lambda|$ large enough. Proposition \ref{prop_q_1_q2} implies that $q_1=q_2$. The proof is complete.

\section{Applications} 

\label{sec_applications}

\subsection{The Laplace operator}
Let us check that the conditions (A1) and (A2) are satisfied for $P=-\Delta$ in $\R^n$, $n\ge 2$. 
Let $\xi\in\R^n$. Then due to the rotational invariance, we may assume that $\xi=(|\xi|,0,\dots,0)$. Let $\lambda<0$ and consider
\[
\zeta_1=(\frac{|\xi|}{2},0,\dots,0)+i(0,\sqrt{\frac{|\xi|^2}{4}+|\lambda|},0,\dots,0)\in \C^n,\quad 
\zeta_2=-\zeta_1.
\]
Hence, $\xi=\zeta_1-\overline{\zeta_2}$, and $\zeta_j\cdot\zeta_j=\lambda$, for $\lambda<0$, $j=1,2$. 

To check the condition (A2), let
\[
L_\zeta(\xi)=(\xi+\zeta)\cdot (\xi+\zeta)-\zeta\cdot\zeta=\xi\cdot\xi+2\xi\cdot\zeta,\quad \xi\in\R^n,\quad \zeta\cdot\zeta=\lambda. 
\]
Then
\[
\p_{\xi_i}L_\zeta(\xi)=2\xi_i+2\zeta_i, \quad \xi\in\R^n,\quad \zeta\cdot\zeta=\lambda,\quad i=1,\dots,n.
\]
The condition $\zeta\cdot\zeta=\lambda<0$ is equivalent to the fact that
\[
\Re\zeta\cdot\Re\zeta-\Im\zeta\cdot\Im\zeta=\lambda,\quad 
\Re\zeta\cdot\Im\zeta=0.
\]
Hence, $|\Re\zeta|^2+|\lambda|=|\Im \zeta|^2$, and therefore, $|\Im\zeta|\ge \sqrt{|\lambda|}$. 
We get
\[
\tilde L_\zeta(\xi)\ge 
\bigg(\sum_{i=1}^n|\p_{\xi_i} L_\zeta(\xi)|^2\bigg)^{1/2}
 \ge 2 \bigg(\sum_{i=1}^n|\Im \zeta_i|^2\bigg)^{1/2}=2|\Im \zeta|Ê\ge 2  \sqrt{|\lambda|},
\]
and thus, the assumption (A2) holds. 

Applying Theorem \ref{thm_main} to the Laplace operator, we recover the standard multidimensional Borg-Levinson Theorem due to \cite{Nachman_Sylveser_Uhlmann_88} and \cite{Nov_1988}, see also 
 \cite{Pai_Ser_2002} for the case of singular potentials.

\subsection{The polyharmonic operator} Generalizing the previous considerations, let us  consider the polyharmonic operator $P=(-\Delta)^m$, $m\ge 2$, in $\R^n$, $n\ge 2$, and show that the conditions (A1) and (A2) are satisfied for this operator.   

To check condition (A1) let us notice that as $\lambda<0$, the fact that $(\zeta\cdot\zeta)^m=\lambda$ is equivalent to the fact that there is an integer $k$, $0\le k\le m-1$, such that 
\begin{equation}
\label{eq_polyharmonic_zeta_1}
|\Re \zeta|^2 -|\Im \zeta|^2=|\lambda|^{1/m}\cos\bigg(\frac{\pi+2\pi k}{m}\bigg),
\end{equation}
\begin{equation}
\label{eq_polyharmonic_zeta_2}
2\Re\zeta \cdot \Im\zeta=|\lambda|^{1/m}\sin\bigg(\frac{\pi+2\pi k}{m}\bigg).
\end{equation}

Let $\xi\in \R^n$, $|\xi|<1$,  be an arbitrary vector. Assuming, as we may,  that $\xi=(|\xi|,0,\dots,0)$, and using \eqref{eq_polyharmonic_zeta_1} and \eqref{eq_polyharmonic_zeta_2}, one can easily see that for $|\lambda|$ large enough, the vectors
\begin{equation}
\label{eq_vectors_poly}
\begin{aligned}
\zeta_1&=\bigg(\frac{|\xi|}{2},\alpha(|\xi|,\lambda),0,\dots,0\bigg)+i(0,\beta(|\xi|,\lambda),0,\dots,0)\in\C^n,\\
\zeta_2&=\bigg(-\frac{|\xi|}{2},\alpha(|\xi|,\lambda),0,\dots,0\bigg)-i(0,\beta(|\xi|,\lambda),0,\dots,0)\in\C^n,
\end{aligned} 
\end{equation}
where
\begin{equation}
\label{eq_vectors_poly_alpha}
\begin{aligned}
\alpha(|\xi|,\lambda)=\frac{1}{\sqrt 2}\bigg(\sqrt{|\lambda|^{2/m}-\frac{1}{2}|\lambda|^{1/m}|\xi|^2\cos\frac{\pi}{m}+\frac{|\xi|^4}{16}}+|\lambda|^{1/m}\cos\frac{\pi}{m}-\frac{|\xi|^2}{4}\bigg)^{1/2},\\
\beta(|\xi|,\lambda)=\frac{1}{\sqrt 2}\bigg(\sqrt{|\lambda|^{2/m}-\frac{1}{2}|\lambda|^{1/m}|\xi|^2\cos\frac{\pi}{m}+\frac{|\xi|^4}{16}}-|\lambda|^{1/m}\cos\frac{\pi}{m}+\frac{|\xi|^2}{4}\bigg)^{1/2},
\end{aligned}
\end{equation}
satisfy $\xi=\zeta_1-\overline{\zeta_2}$ and $P(\zeta_j)=(\zeta_j\cdot \zeta_j)^m=\lambda$, $j=1,2$.

Let us now check the condition (A2).  First we shall show that there exists a constant $C\ge 1$ such that 
\begin{equation}
\label{eq_polyharmonic_zeta_main}
|\Im \zeta|\ge |\lambda|^{1/(2m)}/C,\quad \textrm{for }\zeta\in P^{-1}(\lambda). 
\end{equation}
In  the case when $k$ is such that $\sin(\frac{\pi+2\pi k}{m})=0$,  we have $\cos(\frac{\pi+2\pi k}{m})=-1$, since $\lambda<0$. Thus, \eqref{eq_polyharmonic_zeta_1} implies that $|\Im \zeta|\ge |\lambda|^{1/(2m)}$. 

Consider now the case when $k$ is such that $\sin(\frac{\pi+2\pi k}{m})\ne 0$.  It follows then  from \eqref{eq_polyharmonic_zeta_2}, by an application of the Cauchy--Schwarz inequality,  that 
\begin{equation}
\label{eq_polyharmonic_zeta_3}
|\Re \zeta||\Im \zeta|\ge |\lambda|^{1/m}/C. 
\end{equation}
If $\cos(\frac{\pi+2\pi k}{m})\le 0$, \eqref{eq_polyharmonic_zeta_1} and \eqref{eq_polyharmonic_zeta_3} imply \eqref{eq_polyharmonic_zeta_main}. Assume finally that $\cos(\frac{\pi+2\pi k}{m})> 0$. It follows from \eqref{eq_polyharmonic_zeta_1} and \eqref{eq_polyharmonic_zeta_3} that 
\[
\sqrt{|\Im\zeta|^2 + |\lambda|^{1/m}\cos \bigg(\frac{\pi+2\pi k}{m}\bigg)}|\Im\zeta|\ge |\lambda|^{1/m}/C.
\]
An elementary analysis of this inequality allows us to conclude that the estimate \eqref{eq_polyharmonic_zeta_main} is valid  also in this case.

Let 
\[
L_\zeta(\xi)=((\xi+\zeta)\cdot (\xi+\zeta))^m-\lambda, \quad \xi\in \R^n,\quad (\zeta\cdot \zeta)^m=\lambda<0.
\]
Since the only term in $L_\zeta(\xi)$, which contributes to the derivative $\p^{2m-1}_{\xi_i}L_\zeta(\xi)$, is $(\xi_i+\zeta_i)^{2m}$, we have
\[
\p^{2m-1}_{\xi_i} L_\zeta(\xi)=(2m)!(\xi_i+\zeta_i), \quad i=1,\dots,n. 
\]
Hence, using \eqref{eq_polyharmonic_zeta_main}, we get
\[
\tilde L_\zeta(\xi)\ge \bigg( \sum_{i=1}^n |\p^{2m-1}_{\xi_i} L_\zeta(\xi)|^2\bigg)^{1/2}\ge (2m)!|\Im \zeta|\ge  |\lambda|^{1/(2m)}/C,
\]
which shows that the assumption (A2) holds for the polyharmonic operator. 

We have the following corollary of Theorem \ref{thm_main}.

\begin{cor}
Let $q_1,q_2\in L^\infty(\Omega)$ be real-valued and $\varphi_{k,q_1}$ be an orthonormal basis in $L^2(\Omega)$ of the Dirichlet eigenfunctions of $(-\Delta)^m+q_1$. Furthermore, assume that  the Dirichlet eigenvalues $\lambda_{k,q_j}$ of $(-\Delta)^m+q_j$ satisfy
$
\lambda_{k,q_1}=\lambda_{k,q_2}$,  $k=1,2,\dots$,
and that there exists an orthonormal basis in $L^2(\Omega)$ of the Dirichlet eigenfunctions $\varphi_{k,q_2}$ of $(-\Delta)^m+q_2$ such that
\[
\tilde \gamma\varphi_{k,q_1}=\tilde \gamma\varphi_{k,q_2},\quad k=1,2,\dots. 
\]
Then $q_1=q_2$. 
\end{cor}

\section{The case of incomplete spectral data}

In this section we study the problem of determining the potential from high frequency spectral data. To be precise, let  $\Omega\subset\R^n$, $n\ge 2$, be a bounded smooth domain.
Let $q_1,q_2\in L^\infty(\Omega)$ be real-valued and let $\varphi_{k,q_1}$ be an orthonormal basis in $L^2(\Omega)$
of the Dirichlet eigenfunctions of the operator $P+q_1$.  We assume that there exists an integer $M>0$ such that the eigenvalues $\lambda_{k,q_j}$ of $P+q_j$ satisfy
\begin{equation}
\label{incomplete_lambda}
\lambda_{k,q_1}=\lambda_{k,q_2},\quad \forall k>M,
\end{equation}
and that  there exists an orthonormal basis in $L^2(\Omega)$ of the Dirichlet  eigenfunctions $\varphi_{k,q_2}$ of $P+q_2$ such that
\begin{equation}
\label{incomplete_varphi}
\tilde\gamma\varphi_{k,q_1}=\tilde \gamma\varphi_{k,q_2},\quad  \forall k>M. 
\end{equation}
The problem is whether this still implies that $q_1=q_2$. 
 
Assume that $\lambda_0>0$  is large enough so that for all $\lambda\le -\lambda_0$,  $\lambda\in\rho(P+q_j)$, $j=1,2$, where the self-adjoint operator $P+q_j$ is equipped with the domain $\mathcal{D}(P)$. Let $u_{q_j,f}(\lambda)$ be the solution to the problem 
\begin{align*}
(P+q_j-\lambda)u_{q_j,f}(\lambda)&=0,\quad \textrm{in}\quad \Omega,\\
\gamma u_{q_j,f}(\lambda)&=f,\quad \textrm{on}\quad \p\Omega,
\end{align*}
with $f\in \mathcal{H}^{0,m-1}(\p \Omega)$. 
Setting 
\[
u(\lambda)=u_{q_1,f}(\lambda)-u_{q_2,f}(\lambda),
\]
we get
\begin{equation}
\label{eq_for_green}
\begin{aligned}
(P+q_1-\lambda)u(\lambda)&=(q_1-q_2)u_{q_2,f}(\lambda),\quad \textrm{in}\quad \Omega,\\
\gamma u(\lambda)&=0,\quad \textrm{on}\quad \p\Omega.
\end{aligned}
\end{equation}
Let now $v_1\in H^{2m}(\Omega)$ be such that $(P+q_1-\lambda)v_1=0$ in $\Omega$.
Multiplying \eqref{eq_for_green} by $\overline{v_1}$ and using  Green's formula \eqref{eq_green}, we obtain that
\begin{equation}
\label{eq_after_green_1}
\int_\Omega(q_1-q_2)u_{q_2,f}(\lambda)\overline{v_1}dx=\sum_{i=0}^{2m-1}\int_{\p \Omega} N_{2m-1-i}(u(\lambda))\overline{\p_\nu^i v_1}dS. 
\end{equation}
We would like to choose $v_1$ so that the left hand side of \eqref{eq_after_green_1} vanishes. To this end we are going to require a finite number of orthogonality conditions, to be satisfied by the normal derivatives of $v_1$.

When formulating the orthogonality conditions, notice first that according to \eqref{eq_lambda_der}, we have
\[
\frac{d^l}{d\lambda^l}u_{q_j,f}(\lambda)=l!\sum_{k\ge 1}\frac{1}{(\lambda_{k,q_j}-\lambda)^{l+1}}
\bigg(\sum_{i=0}^{m-1}\int_{\p\Omega} \overline{N_{2m-1-i}(\varphi_{k,q_j})}f_i dS\bigg)
\varphi_{k,q_j},
\]
where $f=(f_0,\dots,f_{m-1})$ and $l$ is large enough, to guarantee the convergence in $H^{2m}(\Omega)$. 
Hence,
\[
\frac{d^l}{d\lambda^l}\p_\nu^r u_{q_j,f}(\lambda)|_{\p\Omega}=l!\sum_{k\ge 1}\frac{1}{(\lambda_{k,q_j}-\lambda)^{l+1}}
\bigg(\sum_{i=0}^{m-1}\int_{\p\Omega} \overline{N_{2m-1-i}(\varphi_{k,q_j})}f_i dS\bigg)
\p_\nu^r\varphi_{k,q_j}|_{\p\Omega},
\]
with convergence in $H^{2m-r-1/2}(\p \Omega)$, $r=m,\dots, 2m-1$.

It follows from \eqref{incomplete_lambda} and \eqref{incomplete_varphi} that
\begin{align*}
\frac{d^l}{d\lambda^l}\p_\nu^r u(\lambda)|_{\p\Omega}&=l!\sum_{k=1}^{M}\frac{1}{(\lambda_{k,q_1}-\lambda)^{l+1}}
\bigg(\sum_{i=0}^{m-1}\int_{\p\Omega} \overline{N_{2m-1-i}(\varphi_{k,q_1})}f_i dS\bigg)
\p_\nu^r\varphi_{k,q_1}|_{\p\Omega}\\
&-l!\sum_{k=1}^{M}\frac{1}{(\lambda_{k,q_2}-\lambda)^{l+1}}
\bigg(\sum_{i=0}^{m-1}\int_{\p\Omega} \overline{N_{2m-1-i}(\varphi_{k,q_2})}f_i dS\bigg)
\p_\nu^r\varphi_{k,q_2}|_{\p\Omega},
\end{align*}
$r=m,\dots, 2m-1$. Integrating $l$ times with respect to $\lambda$, we get
\begin{align*}
\p_\nu^r u(\lambda)|_{\p\Omega}&=\sum_{k=1}^{M}\frac{1}{(\lambda_{k,q_1}-\lambda)}
\bigg(\sum_{i=0}^{m-1}\int_{\p\Omega} \overline{N_{2m-1-i}(\varphi_{k,q_1})}f_i dS\bigg)
\p_\nu^r\varphi_{k,q_1}|_{\p\Omega}\\
&-\sum_{k=1}^{M}\frac{1}{(\lambda_{k,q_2}-\lambda)}
\bigg(\sum_{i=0}^{m-1}\int_{\p\Omega} \overline{N_{2m-1-i}(\varphi_{k,q_2})}f_i dS\bigg)
\p_\nu^r\varphi_{k,q_2}|_{\p\Omega}\\
&+\sum_{k=1}^{l-1}\lambda^k g_{r,k},
\end{align*}
where $g_{r,k}\in H^{2m-r-1/2}(\p \Omega)$, $k=1,\dots, l-1$. Proposition \ref{prop_norm_dirichlet-to-neumann} implies that all $g_{r,k}=0$.
Thus, for any $r=m,\dots, 2m-1$, 
\[
\p_\nu^r u(\lambda)|_{\p\Omega}=\sum_{k=1}^M \alpha_k \p_\nu^r\varphi_{k,q_1}|_{\p\Omega}+\sum_{k=1}^{M}\beta_k \p_\nu^r\varphi_{k,q_2}|_{\p\Omega},
\]
 where the coefficients $\alpha_k$, $\beta_k$ depend on $\lambda$ but do not depend on $r$. 
This together with the fact that $\gamma u(\lambda)=\gamma\varphi_{k,q_j}=0$ implies  that along $\p \Omega$, we have for $ i=0,\dots,m-1$,
\begin{equation}
\label{eq_nor_1}
N_{2m-1-i}(u(\lambda))=\sum_{k=1}^M \alpha_k N_{2m-1-i}(\varphi_{k,q_1}) +\sum_{k=1}^{M}\beta_k N_{2m-1-i}(\varphi_{k,q_2}),
\end{equation}
while for $i=m,\dots,2m-1$,
\begin{equation}
\label{eq_nor_2}
N_{2m-1-i}(u(\lambda))=0.
\end{equation}
It follows from
\eqref{eq_after_green_1}, 
combined with \eqref{eq_nor_1} and \eqref{eq_nor_2}, that in order to have 
\[
\int_\Omega(q_1-q_2)u_{q_2,f}(\lambda)\overline{v_1}dx=0,
\]
we should demand that the $2mM$ orthogonality conditions
\[
\int_{\p \Omega}
 N_{2m-1-i}(\varphi_{k,q_j})\overline{\p_\nu^i v_1}dS=0, \quad k=1,\dots, M,\quad j=1,2,
\]
hold for $i=0,\dots, m-1$.

Thus, in order to conclude that $q_1=q_2$, it suffices to establish that the set
\begin{align*}
\textrm{span}\bigcup_{\lambda<-\lambda_0,\lambda\in\Z}\{
u_{q_1}(\lambda)\overline{u_{q_2}(\lambda)}: u_{q_j}(\lambda)\in H^{2m}(\Omega),(P+q_j-\lambda)u_{q_j}(\lambda)=0 \ \textrm{in}\ \Omega,\\
\int_{\p \Omega}
 N_{2m-1-i}(\varphi_{k,q_j})\overline{\p_\nu^i u_{q_1}(\lambda)}dS=0, k=1,\dots, M, j=1,2, i=0,\dots,m-1 
\}
\end{align*}
is dense in $L^1(\Omega)$, for some $\lambda_0>0$. 

We shall verify the density in the particular case when $P$ is a power of the Laplacian.

\subsection{The Laplace operator} 

In this subsection, we consider the case when $P=-\Delta$, where we recover the following result of  \cite{Isozaki91}.

\begin{thm}
\label{thm_Isozaki}
Let $q_1,q_2\in L^\infty(\Omega)$ be real-valued and $\varphi_{k,q_1}$ be an orthonormal basis in $L^2(\Omega)$
of the Dirichlet eigenfunctions of the operator $-\Delta+q_1$.  Assume that there exists $M>0$ such that the eigenvalues $\lambda_{k,q_j}$ of $-\Delta+q_j$ satisfy
\[
\lambda_{k,q_1}=\lambda_{k,q_2},\quad \forall k>M,
\]
and there exists an orthonormal basis in $L^2(\Omega)$ of the Dirichlet  eigenfunctions $\varphi_{k,q_2}$ of $-\Delta+q_2$ such that
\[
\p_\nu\varphi_{k,q_1}|_{\p\Omega}=\p_\nu\varphi_{k,q_2}|_{\p \Omega},\quad \forall k>M. 
\]
Then $q_1=q_2$.
\end{thm}

In view of the discussion in  the beginning of this section, Theorem \ref{thm_Isozaki} is a direct consequence of Proposition \ref{prop_density_cons_laplace} below. We have learned of the idea of  proving  Theorem \ref{thm_Isozaki} in such a way 
 from the paper \cite{Ramm1993}.   Notice that this method is different from the approach, proposed in the paper  \cite{Isozaki91}.

\begin{prop}
\label{prop_density_cons_laplace} 
 Let $h_k\in L^2(\p \Omega)$, $k=1,\dots, N$, with $N$ being arbitrary but fixed. Then there exists $\lambda_0>0$ such that the space
\begin{align*}
S=\emph{\textrm{span}}&\bigcup_{\lambda<-\lambda_0,\lambda\in\Z}\{
u_{q_1}(\lambda)\overline{u_{q_2}(\lambda)}: u_{q_j}(\lambda)\in H^{2}(\Omega),\\
(-\Delta+q_j-\lambda)u_{q_j}(\lambda)=0 \ \textrm{in}\ & \Omega,
j=1,2, 
\int_{\p \Omega}
 u_{q_1}(\lambda)\overline{h_k}dS=0, k=1,\dots, N 
\}
\end{align*}
is dense in $L^1(\Omega)$. 
\end{prop}

\begin{proof}
Let  $f\in L^\infty(\Omega)$ be such that
\begin{equation}
\label{eq_orth_lap}
\int_\Omega fgdx=0,\quad \forall g\in S. 
\end{equation}
We will show that this implies that  $f=0$.
Let $\xi\in\R^n$ and $\lambda<0$. As before, we may assume that $\xi=(|\xi|,0,\dots,0)$.  Consider
\[
\zeta_1=(\frac{|\xi|}{2},0,\dots,0)+i(0,\sqrt{\frac{|\xi|^2}{4}+|\lambda|},0,\dots,0)\in \C^n,\quad 
\zeta_2=-\zeta_1.
\]
We have $\xi=\zeta_1-\overline{\zeta_2}$, and $\zeta_j\cdot\zeta_j=\lambda$,  $j=1,2$. 
Set
\begin{equation}
\label{eq_eta}
\eta_l=(-\frac{|\xi|}{2}+l,0,\dots,0)+i(0,-\sqrt{\frac{|\xi|^2}{4}+|\lambda|}+\sqrt{l^2 +|\lambda|},0,\dots,0)\in \C^n,
\end{equation}
$l=1,\dots N+1$. Then
\[
\zeta_1+\eta_l=(l,0,\dots,0)+i(0,\sqrt{l^2 +|\lambda|},0,\dots,0), 
\]
and therefore, $(\zeta_1+\eta_l)\cdot (\zeta_1+\eta_l)=\lambda$. 
By Proposition \ref{prop_geometric_optics}, for $|\lambda|$ large enough,  there are solutions
\[
u_{q_1,\lambda,\zeta_1+\eta_l}=e^{i(\zeta_1+\eta_l)\cdot x}(1+w_{\lambda,\zeta_1+\eta_l})\in H^2(\Omega),\quad l=1,\dots,N+1,
\]
to the equation $(-\Delta+q_1-\lambda)u=0$ in $ \Omega$, 
with $\|w_{\lambda,\zeta_1+\eta_l}\|_{L^2( \Omega)}\to 0$ as $\lambda\to-\infty$, and a solution
\[
u_{q_2,\lambda,\zeta_2}=e^{i\zeta_2\cdot x}(1+w_{\lambda,\zeta_2})\in H^2(\Omega),
\]
to $(-\Delta+q_2-\lambda)u=0$ in $ \Omega$, 
with $\|w_{\lambda,\zeta_2}\|_{L^2(\Omega)}\to 0$ as $\lambda\to-\infty$. 
The $N+1$ vectors
\[
H_l=\begin{pmatrix}
\int_{\p\Omega} u_{q_1,\lambda,\zeta_1+\eta_l}\overline{h_1}ds\\
\vdots\\
\int_{\p\Omega} u_{q_1,\lambda,\zeta_1+\eta_l}\overline{h_{N}}ds
\end{pmatrix}\in \C^{N},\quad l=1,\dots,N+1,
\]
are linearly dependent. Thus, there are constants $c_l=c_l(\lambda,\zeta_1+\eta_l)\in \C$, not all equal to zero, such that
\begin{equation}
\label{eq_constraints}
\sum_{l=1}^{N+1}c_l H_l=0. 
\end{equation}
Notice that $c_l$ are independent of $\xi$. 
We can assume that 
$
\sum_{l=1}^{N+1}|c_l|^2=1$,
and that $c_l\to \tilde c_l$, as $\lambda\to -\infty$, $\lambda\in \Z$, where $\tilde c_l\in\C$ are  such that 
\begin{equation}
\label{eq_sum_constants}
\sum_{l=1}^{N+1}|\tilde c_l|^2=1.
\end{equation} 

Let
\[
u_{q_1}(\lambda)=\sum_{l=1}^{N+1}c_l u_{q_1,\lambda,\zeta_1+\eta_l}\in H^2(\Omega), \quad (-\Delta+q_1-\lambda)u_{q_1}(\lambda)=0 \textrm{ in }\Omega.
\]
Then \eqref{eq_constraints} implies that 
\begin{equation}
\label{eq_orth_lap_2}
\int_{\p\Omega} u_{q_1}(\lambda)\overline{h_k}dS=0,\quad k=1,\dots, N. 
\end{equation}
By \eqref{eq_orth_lap} and \eqref{eq_orth_lap_2}, we have
\[
\int_\Omega f\sum_{l=1}^{N+1} c_l e^{i(\xi+\eta_l)\cdot x}(1+w_{\lambda,\zeta_1+\eta_l})(1+\overline{w_{\lambda,\zeta_2}})dx=0,
\]
so that 
\begin{equation}
\label{eq_laplace_des_1}
\int_\Omega f\sum_{l=1}^{N+1} c_l e^{i(\xi+\eta_l)\cdot x}dx=-\int_\Omega f\sum_{l=1}^{N+1} c_l e^{i(\xi+\eta_l)\cdot x}(w_{\lambda,\zeta_1+\eta_l}+\overline{w_{\lambda,\zeta_2}}+w_{\lambda,\zeta_1+\eta_l} \overline{w_{\lambda,\zeta_2}})dx.
\end{equation}

It follows from \eqref{eq_eta} that for all $l=1,\dots, N+1$ and all $x\in \Omega$,
\[
e^{i(\xi+\eta_l)\cdot x}=e^{i\frac{\xi}{2}\cdot x} e^{i l x_1} e^{(\sqrt{\frac{|\xi|^2}{4}+|\lambda|}-\sqrt{l^2+|\lambda|})x_2}\to e^{i\frac{\xi}{2}\cdot x} e^{i l  x_1},\quad \lambda\to-\infty. 
\]
Therefore, using that $\|w_{\lambda,\zeta_1+\eta_l}\|_{L^2(\Omega)}\to 0$, 
$\|w_{\lambda,\zeta_2}\|_{L^2(\Omega)}\to 0$ as $\lambda\to-\infty$, it follows that the right hand
side of \eqref{eq_laplace_des_1}  tends to zero as $\lambda\to-\infty$, $\lambda\in \Z$. By the dominated convergence theorem, we get 
\[
\int_\Omega f \bigg(\sum_{l=1}^{N+1} \tilde c_l e^{i l x_1}\bigg)e^{i\frac{\xi}{2}\cdot x} dx=0,\quad \forall \xi\in\R^n. 
\]
Since $\tilde c_l$ are independent of $\xi$, we conclude that 
\[
f \bigg(\sum_{l=1}^{N+1} \tilde c_l e^{i l x_1}\bigg)=0. 
\]
The function $\sum_{l=1}^{N+1} \tilde c_l e^{i l x_1}$  is real analytic and clearly  does not vanish identically, in view of  \eqref{eq_sum_constants}. 
Thus,  $f=0$ and we are through. 

\end{proof}

\subsection{The biharmonic operator} Let $P=\Delta^2$ be the biharmonic operator in $\R^n$, $n\ge 2$.
The approach of the beginning of this section will be used to prove Theorem \ref{thm_isozaki_bi} in the case $m=2$. 
As in the case of $P=-\Delta$, Theorem \ref{thm_isozaki_bi} is obtained from the following completeness result.

\begin{prop}
\label{prop_bilaplacian_den}
 Let $h_{k,i}\in L^2(\p \Omega)$, $k=1,\dots, N$,  $i=0,1$, with $N$ being arbitrary but fixed. Then there exists $\lambda_0>0$ such that the set
\begin{align*}
S=\emph{\textrm{span}}\bigcup_{\lambda<-\lambda_0,\lambda\in\Z}& \{
u_{q_1}(\lambda)\overline{u_{q_2}(\lambda)}: u_{q_j}(\lambda)\in H^{4}(\Omega),\\
&(\Delta^2+q_j-\lambda)u_{q_j}(\lambda)=0 \ \textrm{in}\  \Omega,
j=1,2, \\
&\int_{\p \Omega}
 \p_\nu^i u_{q_1}(\lambda)\overline{h_{k,i}}dS=0,  k=1,\dots, N, i=0,1 
\}
\end{align*}
is dense in $L^1(\Omega)$. 
\end{prop}

\begin{proof} Let $f\in L^\infty(\Omega)$ be such that 
\[
\int_\Omega fgdx=0,\quad \forall g\in S. 
\]
Let $\xi\in \R^n$ be an arbitrary vector and $\lambda<0$. Again we may assume that $\xi=(|\xi|,0,\dots,0)$. Consider the 
vectors $\zeta_1,\zeta_2\in\C^n$, given by 
\eqref{eq_vectors_poly} for $m=2$, i.e.
\begin{align*}
\zeta_1&=\bigg(\frac{|\xi|}{2},\bigg(\sqrt{\frac{|\lambda|}{4}+\frac{|\xi|^4}{64}}-\frac{|\xi|^2}{8}\bigg)^{1/2},0,\dots,0\bigg)\\
&+
i\bigg(0,\bigg(\sqrt{\frac{|\lambda|}{4}+\frac{|\xi|^4}{64}}+\frac{|\xi|^2}{8}\bigg)^{1/2},0,\dots,0\bigg)\in\C^n,\\
\zeta_2&=\bigg(-\frac{|\xi|}{2},\bigg(\sqrt{\frac{|\lambda|}{4}+\frac{|\xi|^4}{64}}-\frac{|\xi|^2}{8}\bigg)^{1/2},0,\dots,0\bigg)\\
&+
i\bigg(0,-\bigg(\sqrt{\frac{|\lambda|}{4}+\frac{|\xi|^4}{64}}+\frac{|\xi|^2}{8}\bigg)^{1/2},0,\dots,0\bigg)\in\C^n.\\
\end{align*}

We have  $\xi=\zeta_1-\overline{\zeta_2}$ and $P(\zeta_j)=(\zeta_j\cdot\zeta_j)^2=\lambda$, $j=1,2$. 
Set
\begin{align*}
\eta_l=-\zeta_1&+\bigg(l,\bigg(\sqrt{\frac{|\lambda|}{4}+\frac{l^4}{4}}-\frac{l^2}{2}\bigg)^{1/2},0,\dots,0\bigg)\\
&+
i\bigg(0,\bigg(\sqrt{\frac{|\lambda|}{4}+\frac{l^4}{4}}+\frac{l^2}{2}\bigg)^{1/2},0,\dots,0\bigg)\in\C^n,l=1,\dots, 2N+1. \\
\end{align*}
Notice that the sum $\zeta_1+\eta_l$ does not depend on $\xi$ and $P(\zeta_1+\eta_l)=\lambda$. 
It follows from Proposition \ref{prop_geometric_optics} that for $|\lambda|$ large enough, there are solutions
\[
u_{q_1,\lambda,\zeta_1+\eta_l}=e^{i(\zeta_1+\eta_l)\cdot x}(1+w_{\lambda,\zeta_1+\eta_l})\in H^4(\Omega),\quad l=1,\dots,2N+1,
\]
to the equation $(\Delta^2+q_1-\lambda)u=0$ in $\Omega$, 
with $\|w_{\lambda,\zeta_1+\eta_l}\|_{L^2(\Omega)}\to 0$ as $\lambda\to-\infty$, and a solution
\[
u_{q_2,\lambda,\zeta_2}=e^{i\zeta_2\cdot x}(1+w_{\lambda,\zeta_2})\in H^4( \Omega),
\]
to $(\Delta^2+q_2-\lambda)u=0$ in $ \Omega$, 
with $\|w_{\lambda,\zeta_2}\|_{L^2(\Omega)}\to 0$ as $\lambda\to-\infty$. 
Arguing as in Proposition \ref{prop_density_cons_laplace}, one can show that there constants $c_l=c_l(\lambda,\zeta_1+\eta_l)\in \C$ such that
\[
u_{q_1}(\lambda)=\sum_{l=1}^{2N+1}c_l u_{q_1,\lambda,\zeta_1+\eta_l}\in H^{4}(\Omega)
\]
satisfies the conditions
\[
\int_{\p\Omega} \p_\nu^i u_{q_1}(\lambda)\overline{h_{k,i}}dS=0,\quad k=1,\dots, N,\quad i=0,1. 
\]
Moreover, $
\sum_{l=1}^{2N+1}|c_l|^2=1$, $c_l$ do not depend on $\xi$, and $c_l\to \tilde c_l$, as $\lambda\to -\infty$, $\lambda\in \Z$, where $\tilde c_l\in\C$ are  such that 
$
\sum_{l=1}^{2N+1}|\tilde c_l|^2=1.
$

We have
\begin{equation}
\label{eq_laplace_des_2}
\int_\Omega f\sum_{l=1}^{2N+1} c_l e^{i(\xi+\eta_l)\cdot x}(1+w_{\lambda,\zeta_1+\eta_l})(1+\overline{w_{\lambda,\zeta_2}})dx=0. 
\end{equation}
Denote
\begin{align*}
a_l(\lambda)=-\bigg(\sqrt{\frac{|\lambda|}{4}+\frac{|\xi|^4}{64}}-\frac{|\xi|^2}{8}\bigg)^{1/2}+\bigg(\sqrt{\frac{|\lambda|}{4}+\frac{l^4}{4}}-\frac{l^2}{2}\bigg)^{1/2},\\
b_l(\lambda)=
-\bigg(\sqrt{\frac{|\lambda|}{4}+\frac{|\xi|^4}{64}}+\frac{|\xi|^2}{8}\bigg)^{1/2}+\bigg(\sqrt{\frac{|\lambda|}{4}+\frac{l^4}{4}}+\frac{l^2}{2}\bigg)^{1/2},
\end{align*}
when
$l=1,\dots, 2N+1$. As $a_l(\lambda),b_l(\lambda)\to 0$, as $\lambda\to -\infty$, for every fixed $\xi\in\R^n$, we have 
\[
e^{i(\xi+\eta_l)\cdot x}=e^{i\frac{\xi}{2}\cdot x} e^{i l x_1} e^{i a_lx_2}e^{-b_lx_2}\to e^{i\frac{\xi}{2}\cdot x} e^{i l  x_1},\quad \lambda\to-\infty,
\]
when $x\in\Omega$. 
Passing to the limit in \eqref{eq_laplace_des_2} as $\lambda\to-\infty$, $\lambda\in \Z$, we get
\[
\int_\Omega f \bigg(\sum_{l=1}^{2N+1} \tilde c_l e^{i l x_1}\bigg)e^{i\frac{\xi}{2}\cdot x} dx=0,\quad \forall \xi\in\R^n. 
\]
Arguing as in Proposition \ref{prop_density_cons_laplace}, we obtain that  $f=0$. This completes the proof.

\end{proof}

\subsection{The polyharmonic operator} Finally, here, we shall prove  Theorem \ref{thm_isozaki_bi} in the case of an arbitrary polyharmonic operator $P=(-\Delta)^m$, $m\ge 3$.  
As in the previous cases, this  follows from the  following completeness result.

\begin{prop} Let $h_{k,i}\in L^2(\p \Omega)$, $k=1,\dots, N$,  $i=0,\dots,m-1$, with $N$ being arbitrary but fixed. Then there exists $\lambda_0>0$ such that the set
\begin{align*}
S=\emph{\textrm{span}}\bigcup_{\lambda<-\lambda_0,\lambda\in\Z}& \{
u_{q_1}(\lambda)\overline{u_{q_2}(\lambda)}: u_{q_j}(\lambda)\in H^{2m}(\Omega),\\
&((-\Delta)^m+q_j-\lambda)u_{q_j}(\lambda)=0 \ \textrm{in}\  \Omega,
j=1,2, \\
&\int_{\p \Omega}
 \p_\nu^i u_{q_1}(\lambda)\overline{h_{k,i}}dS=0,  k=1,\dots, N, i=0,\dots,m-1
\}
\end{align*}
is dense in $L^1(\Omega)$. 
\end{prop}

\begin{proof} 
Let $f\in L^\infty(\Omega)$ be such that 
\[
\int_\Omega fgdx=0,\quad \forall g\in S.
\]
Let $\xi\in\R^n$, $|\xi|<1$, be an arbitrary vector and let $\lambda<0$ with $|\lambda|$ large enough. 
Consider the vectors $\zeta_1, \zeta_2\in \C^n$, given by \eqref{eq_vectors_poly}. 
For $l=1,\dots, mN+1$, we  introduce
\[
\eta_l=-\zeta_1+(l,\alpha(2l,\lambda),0,\dots,0)+i(0,\beta(2l,\lambda),0,\dots,0)\in \C^n, 
\]
where $\alpha(2l,\lambda)$ and $\beta(2l,\lambda)$ are defined by \eqref{eq_vectors_poly_alpha}.

One can easily see that for every $\xi\in \R^n$ fixed, 
\begin{align*}
\alpha(2l,\lambda)-\alpha(|\xi|,\lambda)\to 0,\text{ as }\lambda\to -\infty,\\
\beta(2l,\lambda)-\beta(|\xi|,\lambda)\to 0,\text{ as }\lambda\to -\infty. 
\end{align*} 
The choice of the vectors $\zeta_1$, $\zeta_2$ and $\eta_l$ allows us to repeat the proofs of Proposition \ref{prop_density_cons_laplace} and Proposition \ref{prop_bilaplacian_den} to conclude that $f=0$.  The proof is complete.

\end{proof}

\section{Acknowledgements}  
We would like to thank Peter Kuchment and Valeriy Serov for stimulating discussions. 
The research of K.K. was financially supported by the
Academy of Finland (project 125599). 
The research of L.P. was financially supported by Academy of Finland Center of Excellence programme 213476. 
The writing of this paper was completed at  the
Mathematical Sciences Research Institute,  Berkeley, whose hospitality is gratefully acknowledged.

\end{document}